\documentclass[11pt]{amsart}
\usepackage{graphicx}
\usepackage{amssymb}
\usepackage{amsmath}
\usepackage{amsthm}
\usepackage{bm}
\usepackage{placeins}
\newtheorem{tw}{Theorem} 
\newtheorem*{Mtw}{Main Theorem}

\newtheorem{wn}[tw]{Corollary}

\theoremstyle{remark}
\newtheorem{uw}[tw]{Remark}

\theoremstyle{definition}

\newcommand{\bez}{\setminus}

\newcommand{\fii}{\varphi}

\newcommand{\ro}{\varrho}

\newcommand{\gen}[1]{\langle #1 \rangle}
\newcommand{\map}[3]{#1\colon #2\to #3}
\newcommand{\field}[1]{\mathbb{#1}}

\newcommand{\rr}{\field{R}}

\newcommand{\st}{\;|\;}

\begin{document}

\numberwithin{equation}{section}
\title{Roots of crosscap slides and crosscap transpositions}

\author{Anna Parlak \hspace{1.5em}  Micha\l $\ $Stukow}


\address[]{
Institute of Mathematics, Faculty of Mathematics, Physics and Informatics, University of Gda\'nsk, 80-308 Gda\'nsk, Poland
}

\thanks{The second author is supported by grant 2015/17/B/ST1/03235 of National Science Centre, Poland.}

\email{anna.parlak@gmail.com, trojkat@mat.ug.edu.pl}


\begin{abstract}
Let $N_{g}$ denote a closed nonorientable surface of genus $g$. For $g \geq 2$ the mapping class group $\mathcal{M}(N_{g})$ is generated by Dehn twists and one crosscap slide ($Y$--homeomorphism) or by Dehn twists and a crosscap transposition.
Margalit and Schleimer 
observed that Dehn twists have nontrivial roots. 
We give necessary and sufficient conditions for the existence of roots of crosscap slides and crosscap transpositions.
\end{abstract}

\maketitle%

 \section{Introduction}
 Let $N_{g, s}^{n}$ be a connected, nonorientable surface of genus $g$ with $s$ boundary components and $n$ punctures, that is a surface obtained from a connected sum of $g$ projective planes $N_{g}$ by removing $s$ open disks and specifying the set $\Sigma = \lbrace p_{1}, \ldots, p_{n} \rbrace $ of $n$
 distinguished points in the interior of $N_{g}$.
 If $s$ or/and $n$ equals zero, we omit it from notation. The \emph{mapping class group} $\mathcal{M}(N_{g, s}^{n})$ consists of isotopy classes of self--homeomorphisms $h: N_{g,s}^{n} \rightarrow N_{g,s}^{n}$ fixing boundary components pointwise and such that $h(\Sigma) = \Sigma$.  The mapping class group $\mathcal{M}(S_{g,s}^{n})$ of an orientable surface is defined analogously, but we consider only orientation--preserving maps. If we allow orientation--reversing maps, we obtain the \emph{extended mapping class group} $\mathcal{M}^{\pm}(S_{g,s}^{n})$. By abuse of notation, we identify a homeomorphism with its isotopy class. 
 
 In the orientable case, the mapping class group $\mathcal{M}(S_{g})$ is generated by Dehn twists \cite{Lick1}. As for nonorientable surfaces, Lickorish proved that Dehn twists alone do not generate $\mathcal{M}(N_{g})$, $g \geq 2$. This group is generated by Dehn twists and one crosscap slide ($Y$--homeomorphism)~\cite{Lick3}.
 A presentation for $\mathcal{M}(N_{g})$ using these generators was obtained by Stukow \cite{StukowSimpSzepPar}. This presentation was derived from the presentation given by Paris and Szepietowski \cite{SzepParis}, which used as generators Dehn twists and yet another homeomorphisms of nonorientable surfaces, so--called crosscap transpositions. 
 
  
 
 Margalit and Schleimer discovered a surprising property of Dehn twists: in the mapping class group of a closed, connected, orientable surface $S_{g}$ of genus $g \geq 2$, every Dehn twist has a nontrivial root \cite{MargSchleim}. It is natural to ask if crosscap slides and crosscap transpositions also have a similar property. The main goal of this paper is to prove the following:
 \begin{Mtw}
In $\mathcal{M}(N_{g})$ a nontrivial root of a crosscap transposition [or a crosscap slide] exists
if and only if $g \geq 5$ or $g = 4$ and the complement of the support of this crosscap transposition [or crosscap slide] is orientable. \label{main}
\end{Mtw}
 
%
%
 
 \section{Preliminaries}
 \subsection*{Crosscap transpositions and crosscap slides.}
 Let $N = N_{g}$ be a nonorientable surface of genus $g \geq 2$. Let $\alpha$ and $\mu$ be two simple closed curves in $N$ intersecting in one point, such that $\alpha$ is two--sided and $\mu$ is one--sided. A regular neighborhood of $\mu \cup \alpha$ is homeomorphic to the Klein bottle with a hole denoted by $K$. A convenient model of $K$ is a disk with 2 crosscaps, see Figure \ref{UY}. In this figure shaded disks represent crosscaps, thus the boundary points of these disks are identified by the antipodal map.
 \begin{figure}[h]
 	\begin{center}
 		\includegraphics[width=0.95\textwidth]{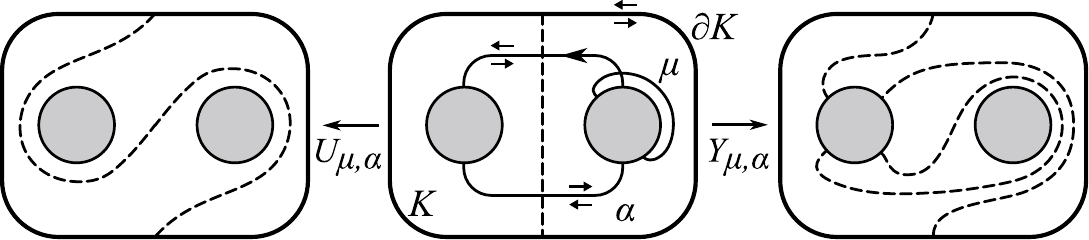}
 		\caption{Crosscap transposition and crosscap slide.}
 		\label{UY} 
 	\end{center}
 \end{figure}
 
 A \emph{crosscap transposition} $U_{\mu,\alpha}$ specified by $\mu$ and $\alpha$ is 
 a homeomorphism of $K$ which interchanges two crosscaps keeping the boundary of $K$ fixed \cite{SzepParis}. It may be extended by the identity to a homeomorphism of $N$. If $t_{\alpha}$ is the Dehn twist about $\alpha$ (with the direction of the twist indicated by small arrows in Figure \ref{UY}), then $Y_{\mu,\alpha} = t_{\alpha}U_{\mu,\alpha}$ is a \emph{crosscap slide} of $\mu$ along $\alpha$, that is the effect of pushing $\mu$ once along $\alpha$ keeping the boundary of $K$ fixed. Note that $U_{\mu,\alpha}^{2}=Y_{\mu,\alpha}^{2}=t_{\partial K}$. 
 
 
%

  

If $g$ is even, then the complement of $K$ in $N_g$ can be either a nonorientable surface $N_{g-2,1}$ or an orientable surface $S_{\frac{g-2}{2},1}$, therefore on surfaces of even genus two conjugacy classes of crosscap slides and crosscap transpositions exist.

 \subsection*{Notation}
 Represent $N_g$ as a connected sum of $g$ projective planes and let $\mu_1,\ldots,\mu_{g}$ be 
 one--sided circles that correspond to crosscaps as in indicated in Figure \ref{fig_surface}.
 By abuse of notation, we identify $\mu_i$ with the corresponding crosscap.
 \begin{figure}[h]
 	\begin{center}
 		\includegraphics[width=0.8\textwidth]{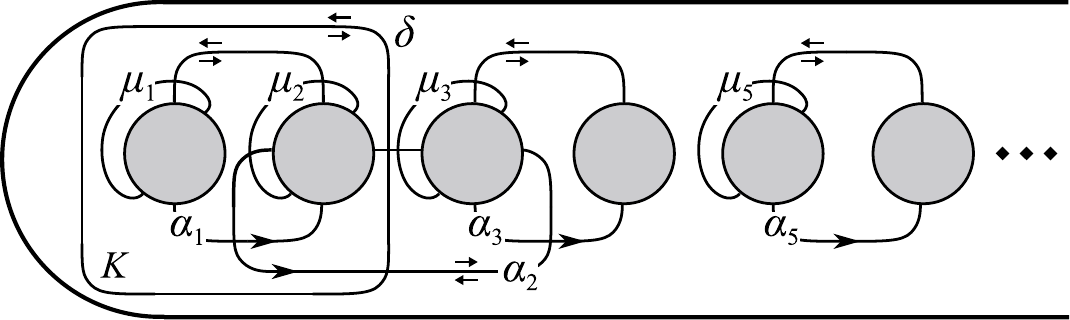}
 		\caption{Nonorientable surface $N_g$.}
 		\label{fig_surface} 
 	\end{center}
 \end{figure}

 If $\alpha_1,\ldots,\alpha_{g-1}$ are two--sided circles indicated in the same figure, then for each $i=1,\ldots,g-1$ by $t_{\alpha_i},u_i,y_i$ we denote the Dehn twist about $\alpha_i$, the crosscap transposition $U_{\mu_{i+1},\alpha_i}$, and the crosscap slide $Y_{\mu_{i+1},\alpha_i}$, respectively.

 
\subsection*{Relations in the mapping class group of a nonorientable surface} 
A full presentation for $\mathcal{M}(N_{g})$ is given in \cite{SzepParis, StukowSimpSzepPar}. Among others, the following relations hold in $\mathcal{M}(N_{g})$:
\begin{itemize}
	\item[(R1)] $u_{i}u_{j} = u_{j}u_{i}$ for $i,j = 1, \ldots, g-1, \ |i-j|>1$,
	\item[(R2)] $u_{i}u_{i+1}u_{i} = u_{i+1}u_{i}u_{i+1}$ for $i=1,\ldots, g-2$,
	\item[(R3)] $(u_{1} \ldots u_{g-1})^{g} = 1$,
	\item[(R4)] $t_{\alpha_{i}}u_{j} = u_{j}t_{\alpha_{i}}$ and hence $y_{i}u_{j} = u_{j}y_{i}$ for $i,j=1, \ldots, g-1, \ |i-j|>1$ 
\end{itemize}
It is straightforward to check that relations (R1)--(R3) imply
\begin{itemize}
 \item[(R5)] $(u_{1}^{2}u_{2}\ldots u_{g-1})^{g-1} = 1$
\end{itemize}
Geometrically $u_{1}\ldots u_{g-1}$ is a cyclic rotation of $\mu_1,\ldots,\mu_{g}$ and $u_{1}^{2}u_{2}\ldots u_{g-1}$ is a cyclic rotation of $\mu_2,\ldots,\mu_{g}$ around $\mu_1$. In particular, 
\begin{itemize}
 \item[(R6)] $(u_{1} \ldots u_{g-1})^{g} = (u_{1}^{2}u_{2} \ldots u_{g-1})^{g-1} = t_{\partial N_{g,1}}$ in $\mathcal{M}(N_{g,1})$.
\end{itemize}

We also have the following chain relation between Dehn twists (Proposition 4.12 of \cite{MargaliFarb}): if $k \geq 2$ is even and $c_{1}, \ldots, c_{k}$ is a chain of simple closed curves in a surface $S$, such that the boundary of a closed regular neighborhood of their union is isotopic to $d$, then 
\begin{itemize}
\item[(R7)] $(t_{c_{1}}\ldots t_{c_{k}})^{2k+2} = t_{d}$.
 \end{itemize}
\section{Proof of the Main Theorem}
\begin{uw}\label{rem:odd}
	Automorphisms of $H_{1}(N_{g}; \mathbb{R})$ induced by crosscap transpositions and crosscap slides have determinants equal to $-1$, so if a root of a crosscap slide or a crosscap transposition exists, it must be of odd degree. 
\end{uw}
Let $K$ be a subsurface of $N_g$ that is a Klein bottle with one boundary component $\delta$ and which contains $\mu_1$ and $\mu_2$ (Figure \ref{fig_surface}). In particular $u_1^2=y_1^2=t_{\delta}$.

\subsection*{The case of ${g\geq 5}$ odd.}
Let $p,q \in \mathbb{Z}$ be such that $2p + q(g-2) = 1$. By relations (R6) and (R1),
%
\[\begin{aligned}
   u_{1}^{2}&=t_\delta=(u_{3} \ldots u_{g-1})^{g-2}\\
   u_{1}^{2p}&=(u_{3} \ldots u_{g-1})^{p(g-2)}\\
   u_{1}&=\left((u_{3}\ldots u_{g-1})^{p}u_{1}^{q}\right)^{g-2}
  \end{aligned}
\]
Analogously, by relations (R6), (R1) and (R4), $y_{1} = \left((u_{3}\ldots u_{g-1})^{p}y_{1}^{q}\right)^{g-2}$.

\subsection*{The case of ${g \geq 6}$ even and ${N_{g} \backslash K}$ nonorientable.}
Let $p,q \in \mathbb{Z}$ be such that $2p + q(g-3) = 1$. By relations (R6) and (R1), 
\[\begin{aligned}
u_{1}^{2}&=t_\delta=(u_{3}^{2}u_{4} \ldots u_{g-1})^{g-3}\\
u_{1}^{2p}&=(u_{3}^{2}u_{4} \ldots u_{g-1})^{p(g-3)}\\
u_{1}&=((u_{3}^{2}u_{4} \ldots u_{g-1})^{p}u_{1}^{q})^{g-3}.
\end{aligned}\]
Analogously, by relations (R6), (R1) and (R4), $y_{1} = ((u_{3}^{2}u_{4} \ldots u_{g-1})^{p}y_{1}^{q})^{g-3}$.

\subsection*{The case of ${g \geq 4}$ even and ${N_{g} \backslash K}$ orientable.}
Suppose now that crosscap transposition $u$ and crosscap slide $y$ are supported in a Klein bottle with a hole $K$ such that $N_g\bez K$ is orientable. 
If $c_{1}, \ldots, c_{g-2}$ is a chain of two--sided circles in $N_g\bez K$, then by relation (R7),
\[\begin{aligned}
u_{1}^{2}&=t_{\partial K}=(t_{c_{1}}\ldots t_{c_{g-2}})^{2g-2}\\
\left(u_{1}^{2}\right)^{\frac{g}{2}}&=\left((t_{c_{1}}\ldots t_{c_{g-2}})^{2g-2}\right)^{\frac{g}{2}}\\
u_{1}&=((t_{c_{1}}\ldots t_{c_{g-2}})^{g}u_{1}^{-1})^{g-1}.
\end{aligned}\]
Analogously, $y_{1} = ((t_{c_{1}}\ldots t_{c_{g-2}})^{g}y_{1}^{-1})^{g-1}$.
\subsection*{The case of $g=2$.}
	Crosscap slides and a crosscap transpositions are primitive in $\mathcal{M}(N_{2})$ because \cite {Lick3}
	\[\begin{aligned}
	\mathcal{M}(N_{2})&\cong \langle t_{\alpha_{1}}, y_{1}  \st  t_{\alpha_{1}}^{2} = y_{1}^{2} = (t_{\alpha_{1}}y_{1})^{2} = 1 \rangle\\
	&\cong \langle t_{\alpha_{1}}, u_{1}  | \ t_{\alpha_{1}}^{2} = u_{1}^{2} = (t_{\alpha_{1}}u_{1})^{2} = 1 \rangle \cong \mathbb{Z}_{2} \oplus \mathbb{Z}_{2}.   
	  \end{aligned}\] 
\subsection*{The case of $g=3$.} 
\begin{uw}\label{rem:order6}
It is known that the mapping class group $\mathcal{M}(N_{3})$ is hyperelliptic \cite{Stukow_HiperOsaka} and has the central element $\ro$ such that $\mathcal{M}(N_{3})/\gen{\ro}$ is the extended mapping class group $\mathcal{M}^{\pm}(S_{0}^{3,1})$ of a sphere with 4 punctures. Two upper subscripts mean that we have four punctures on the sphere, but one of them must be fixed. This implies \cite{Buskirk} that the maximal finite order of an element in $\mathcal{M}^{\pm}(S_{0}^{3,1})$ is 3, and hence the maximal finite order of an element in $\mathcal{M}(N_{3})$ is 6. Moreover, each two rotations of order 3 in $\mathcal{M}^{\pm}(S_{0}^{3,1})$ are conjugate, which easily implies that each two elements of order 6 in $\mathcal{M}(N_{3})$ are conjugate. The details of the proof of the last statement are completely analogous to that used in \cite{MaxHyp}, hence we skip them.

The same conclusion can be obtained also purely algebraically: it is known \cite{Scharlemann} that $\mathcal{M}(N_{3})\cong\mathrm{GL}(2,\mathbb{Z})$ and the maximal finite order of an element in $\mathrm{GL}(2,\mathbb{Z})$ is 6. Moreover, there is only one conjugacy class of such elements in $\mathrm{GL}(2,\mathbb{Z})$ --- for details see for example Theorem 2 of \cite{Meskin}.
\end{uw}
We will show that crosscap transpositions do not have nontrivial roots in $\mathcal{M}(N_{3})$. 
Suppose that $x \in \mathcal{M}(N_{3})$ exists such that $x^{2k+1} = u_{1}$, where $k\geq 1$ (see Remark \ref{rem:odd}). Then 
\[x^{4k+2}=u_1^2=t_\delta=1.\]
By Remark \ref{rem:order6}, $k=1$. Moreover, by relation (R7),
\[(t_{\alpha_1}t_{\alpha_2})^6=t_\delta=1,\]
hence $x$ is conjugate to $t_{\alpha_1}t_{\alpha_2}$. This contradicts Remark \ref{rem:odd}, because Dehn twists induce automorphisms of $H_1(N_3;\rr)$ with determinant equal to 1 and $x^3=u_{1}$.

In the case of a crosscap slide the argument is completely analogous, hence we skip the details.

\subsection*{The case of $g=4$ and ${N_{4} \backslash K}$ nonorientable.}
If $N_{4} \backslash K$ is nonorientable, then $\delta$ cuts $N_4$ into two Klein bottles with one boundary component: $K$ and $K_1$. Moreover, as was shown in Appendix A of \cite{Stukow_twist},
\[\begin{aligned}
   \mathcal{M}(K)&=\langle t_{\alpha_{1}}, u_{1} \ | \ u_{1}t_{\alpha_{1}} = t_{\alpha_{1}}^{-1}u_{1}\rangle\\
   \mathcal{M}(K_1)&=\langle t_{\alpha_{3}}, u_{3} \ | \ u_{3}t_{\alpha_{3}} = t_{\alpha_{3}}^{-1}u_{3}\rangle.
  \end{aligned}
\]
If $x \in \mathcal{M}(N_{4})$ exists such that $x^{2k+1} = u_{1}$ and $k\geq 1$ (see Remark \ref{rem:odd}), then \[x^{4k+2}=u_1^2=t_\delta.\]
In particular, $x$ commutes with $t_\delta$ and
\[t_\delta=xt_{\delta}x^{-1}=t_{x(\delta)}^{\pm}.\]
By Proposition 4.6 of \cite{Stukow_twist}, up to isotopy of $N_4$, $x(\delta)=\delta$. Because $u_1$ does not interchange two sides of $\delta$ and does not reverse the orientation of $\delta$, $x$ has exactly the same properties. Therefore, we can assume that $x$ is composed of maps of $K$ and $K_1$. Moreover $u_1=x^{2k+1}$ interchanges $\mu_1$ and $\mu_2$ and does not interchange $\mu_3$ and $\mu_4$, hence 
\[\begin{aligned}
  x&=t_{\alpha_{1}}^{k_{1}}u_{1}^{2m_{1}+1}t_{\alpha_{3}}^{k_{2}}u_{3}^{2m_{2}}=
t_{\alpha_{1}}^{k_{1}}u_{1}t_{\alpha_{3}}^{k_{2}}t_{\delta}^{m_1+m_2}\\
x^2&=t_{\alpha_{3}}^{2k_{2}}t_{\delta}^{2m_1+2m_2+1}
\end{aligned}\]
But then 
\[t_\delta=(x^2)^{2k+1}=t_{\alpha_{3}}^{2k_{2}(2k+1)}t_{\delta}^{(2m_1+2m_{2}+1)(2k+1)}\]
which is a contradiction, because Dehn twists about disjoint circles generate a free abelian group (Proposition 4.4 of \cite{Stukow_twist}).

In the case of a crosscap slide the argument is completely analogous, hence we skip the details.
%
%
%
\section{Roots of elementary braids in the mapping class group of $n$-punctured sphere.}
 Margalit and Schleimer observed in \cite{MargSchleim} that if $g\geq 5$, then roots of elementary braids in $\mathcal{M}(S_{0}^g)$ exist. The Main Theorem implies slightly stronger version of that result. 

\begin{wn}
	An elementary braid in the mapping class group $\mathcal{M}(S_{0}^n)$ or in the extended mapping class group $\mathcal{M}^\pm(S_{0}^n)$	
	of \mbox{$n$-punctured} sphere has a nontrivial root if and only if $n \geq 5$. 
\end{wn}
\begin{proof}
 By Proposition 2.4 of \cite{SzepParis}, there is a monomorphism
\[\map{\fii}{\mathcal{M}^{\pm}(S_{0}^g)}{\mathcal{M}(N_g)}\]
which is induced by blowing up each puncture to a crosscap. In particular, this monomorphism sends elementary braids to crosscap transpositions. Moreover, all roots of crosscap transpositions constructed in the proof of the Main Theorem are elements of $\fii(\mathcal{M}(S_{0}^g))$.
\end{proof}
\bibliographystyle{abbrv}

\begin{thebibliography}{10}

\bibitem{MargaliFarb}
B.~Farb and D.~Margalit.
\newblock {\em A Primer on Mapping Class Groups}, volume~49 of {\em Princeton
  Mathematical Series}.
\newblock Princeton Univ. Press, 2011.

\bibitem{Buskirk}
R.~Gillette and J.~{V}an Buskirk.
\newblock The word problem and consequences for the braid groups and
  mapping-class groups of the 2-sphere.
\newblock {\em Trans. Amer. Math. Soc.}, 131:277--296, 1968.

\bibitem{Lick1}
W.~B.~R. Lickorish.
\newblock A representation of orientable combinatorial 3-manifolds.
\newblock {\em Ann. of Math.}, 76:531--540, 1962.

\bibitem{Lick3}
W.~B.~R. Lickorish.
\newblock Homeomorphisms of non--orientable two--manifolds.
\newblock {\em Math. Proc. Cambridge Philos. Soc.}, 59:307--317, 1963.

\bibitem{MargSchleim}
D.~Margalit and S.~Schleimer.
\newblock Dehn twists have roots.
\newblock {\em Geom. Topol.}, 13(3):1495--1497, 2009.

\bibitem{Meskin}
S.~Meskin.
\newblock {\em Periodic automorphisms of the two--generator free group}, volume
  372 of {\em Lecture {N}otes in {M}ath.}, pages 494--498.
\newblock Springer--{V}erlag, 1974.

\bibitem{SzepParis}
L.~Paris and B.~Szepietowski.
\newblock A presentation for the mapping class group of a nonorientable
  surface.
\newblock {\em Bull. Soc. Math. France.}, 143:503--566, 2015.

\bibitem{Scharlemann}
M.~Scharlemann.
\newblock The complex of curves on non-orientable surfaces.
\newblock {\em J. London Math. Soc.}, 25(2):171--184, 1982.

\bibitem{MaxHyp}
M.~{S}tukow.
\newblock Conjugacy classes of finite subgroups of certain mapping class
  groups.
\newblock {\em Turkish J. Math.}, 28(2):101--110, 2004.

\bibitem{Stukow_twist}
M.~Stukow.
\newblock Dehn twists on nonorientable surfaces.
\newblock {\em Fund. Math.}, 189:117--147, 2006.

\bibitem{StukowSimpSzepPar}
M.~{S}tukow.
\newblock A finite presentation for the mapping class group of a nonorientable
  surface with {D}ehn twists and one crosscap slide as generators.
\newblock {\em J. Pure Appl. Algebra}, 218(12):2226--2239, 2014.

\bibitem{Stukow_HiperOsaka}
M.~Stukow.
\newblock A finite presentation for the hyperlliptic mapping class group of a
  nonrientable surface.
\newblock {\em Osaka J. Math.}, 52(2):495--515, 2015.

\end{thebibliography}

\end{document}